\documentclass[10pt]{article}

 \usepackage{amsmath,amsthm,amssymb}
 \usepackage{makeidx,epsfig,lscape}
 \usepackage{color,colortbl}
 \usepackage{fancyhdr}
 \usepackage{xcolor,pict2e}
 \usepackage[latin1]{inputenc}
 
 \usepackage{lipsum}% just to automatically generate text
\makeatletter
\newcommand\ackname{Acknowledgements}
\if@titlepage
  \newenvironment{acknowledgements}{%
      \titlepage
      \null\vfil
      \@beginparpenalty\@lowpenalty
      \begin{center}%
        \bfseries \ackname
        \@endparpenalty\@M
      \end{center}}%
     {\par\vfil\null\endtitlepage}
\else
  \newenvironment{acknowledgements}{%
      \if@twocolumn
        \section*{\abstractname}%
      \else
        \small
        \begin{center}%
          {\bfseries \ackname\vspace{-.5em}\vspace{\z@}}%
        \end{center}%
        \quotation
      \fi}
      {\if@twocolumn\else\endquotation\fi}
\fi
 
\newcommand{\R}{\mathbb R}

\newcommand{\Qv}{\mathbb Q}
\newcommand{\Pv}{\mathbb P}
\newcommand{\E}{\mathbb E}

\newcommand{\Dv}{\mathbb D}

 \renewcommand{\headrulewidth}{0pt}
 \renewcommand{\footrulewidth}{0.5pt}

 \definecolor{myaqua}{rgb}{0.0,0.5,0.55}
 \definecolor{lightaqua}{rgb}{0.75,0.95,0.95}

 \usepackage[colorlinks = true,
            linkcolor = myaqua,
            urlcolor  = blue,
            citecolor = red,
            pdfpagemode=UseOutlines,
            bookmarksnumbered=true,
            pdfpagelabels,
            breaklinks]{hyperref}

\usepackage{caption}
\usepackage{floatrow}
\newtheorem{theorem}{Theorem}
\newtheorem{prop}[theorem]{Proposition}
\newtheorem{lem}{Lemma}
\newtheorem{coro}{Corollary}
\newtheorem{defn}{Definition}[section]
\newtheorem{rem}{Remark}[section]
\def\lin#1#2{\textcolor[rgb]{0.6,0.6,0.6}{\vspace*{#1mm} \hrule
   height 3 pt \vspace*{#2mm}}}
\def\bt{\begin{tabular}}
\def\et{\end{tabular}}
\def\and{\mbox{ and }}
\def\E{\mbox{\bf E}}

\def\1{{\bf 1}}

 \def\boxx#1#2#3#4#5{
 {\linethickness{#4pt}\put(#1,#5){\color{myaqua}{\line(1,0){#3}}}}
 \multiput(#1,#2)(0,#4){2}{\line(1,0){#3}}
 \multiput(#1,#2)(#3,0){2}{\line(0,1){#4}}
  }

\begin{document}

% \fancyhead[L]{\hspace*{-13mm}
% \bt{l}{\bf Open Journal of *****, 2014, *,**}\\
% Published Online **** 2014 in SciRes.
% \href{http://www.scirp.org/journal/*****}{\color{blue}{\underline{\smash{http://www.scirp.org/journal/****}}}} \\
% \href{http://dx.doi.org/10.4236/****.2014.*****}{\color{blue}{\underline{\smash{http://dx.doi.org/10.4236/****.2014.*****}}}} \\
% \et}
% \fancyhead[R]{\special{psfile=pic1.ps hoffset=366 voffset=-33}}

 $\mbox{ }$

 \vskip 12mm

{ % \fontfamily{Cambria}\selectfont

% "Title of the Paper"
{\noindent{\Large\bf\color{myaqua}
  On the study of processes of \texorpdfstring{$\Sigma(H)$}{sigma(H)} and \texorpdfstring{$\sum_{s}(H)$}{sigma(s)(H)} classes}} %\mathbb{R}_{+} $\Sigma(H)$
%
% \runtitle{Change-Point Analysis of Survival Data}
\\[6mm]
{\bf Fulgence EYI-OBIANG$^1$, Youssef OUKNINE$^2$, Octave MOUTSINGA$^3$}}
\\[2mm]
{ %\fontfamily{Calibri}\selectfont
 $^1$Université Cady Ayyad de Marrakech and Université des Sciences et Techniques de Masuku 
 \\
Email: \href{mailto:feyiobiang@yahoo.fr}{\color{blue}{\underline{\smash{feyiobiang@yahoo.fr}}}}\\[1mm]
$^2$Université Cady Ayyad de Marrakech and Hassan II Academy of Sciences and Technologies Rabat\\
Email:
\href{mailto:ouknine@ucam.ac.ma}{\color{blue}{\underline{\smash{ouknine@ucam.ac.ma}}}}\\[1mm]
$^3$ Université des Sciences et Techniques de Masuku\\
\href{mailto:octavemoutsing-pro@yahoo.fr}{\color{blue}{\underline{\smash{octavemoutsing-pro@yahoo.fr}}}}\\[1mm]
\\[1mm]
%\href{mailto:octavemoutsing-pro@yahoo.fr}{\color{blue}{\underline{\smash{octavemoutsing-pro@yahoo.fr}}}}\\[1mm]
\lin{5}{7}

 {  %\fontfamily{Cambria}\selectfont
 {\noindent{\large\bf\color{myaqua} Abstract}{\bf \\[3mm]
 \textup{
In papers by Yor, a remarkable class $(\Sigma)$ of submartingales is introduced, which, up to technicalities, are submartingales $(X_{t})_{t\geq0}$ whose increasing process is carried by the times $t$ such that $X_{t}=0$. These submartingales have several applications in stochastic analysis: for example, the resolution of Skorokhod embedding problem, the study of Brownian local times and the study of zeros of continuous martingales. The submartingales of class $(\Sigma)$ have been extensively studied in a series of articles by Nikeghbali (part of them in collaboration with Najnudel, some others with Cheridito and Platen). On the other hand, stochastic calculus has been extended to signed measures by Ruiz de Chavez \cite{chav} and Beghdadi-Sakrani \cite{sak}. In \cite{f}, the authors of the present paper have extended the notion of submartingales of class $(\Sigma)$ to the setting of Ruiz de Chavez \cite{chav} and Beghdadi-Sakrani \cite{sak}, giving two different classes of stochastic processes named classes $\sum(H)$ and $\sum_{s}(H)$ where from tools of the theory of stochastic calculus for signed measures, the authors provide general frameworks and methods for dealing with processes of these classes. In this work, we first give some formulas of multiplicative decomposition for processes of these classes. Afterward, we shall establish some representation results allowing to recover any process of one of these classes from its final value and the last time it visited the origin. 
 }}}
 \\[4mm]
 {\noindent{\large\bf\color{myaqua} Keywords}{\bf \\[3mm]
 Stochastic calculus for signed measures; zeros of continuous martingales; class $\sum$; class $\sum(H)$; class $\sum_{s}(H)$; last passage times
}}}
\lin{3}{1}

\renewcommand{\headrulewidth}{0.5pt}
\renewcommand{\footrulewidth}{0pt}

 \pagestyle{fancy}
 \fancyfoot{}
 \fancyhead{} % clear all header and footer fields
 \fancyhf{}
 \fancyhead[RO]{\leavevmode \put(-90,0){\color{myaqua}F. EYI-OBIANG et al} \boxx{15}{-10}{10}{50}{15} }
 %\fancyhead[LE]{\leavevmode \put(0,0){\color{myaqua}F. EYI-OBIANG et al (2015)}  \boxx{-45}{-10}{10}{50}{15} }
 \fancyfoot[C]{\leavevmode
 %\put(0,0){\color{lightaqua}\circle*{34}}
 %\put(0,0){\color{myaqua}\circle{34}}
 \put(-2.5,-3){\color{myaqua}\thepage}}

 \renewcommand{\headrule}{\hbox to\headwidth{\color{myaqua}\leaders\hrule height \headrulewidth\hfill}}
\section{Introduction}

{ \fontfamily{times}\selectfont
 \noindent % L'introduction commence ici

From stochastic calculus for signed measures, we have introduced  in \cite{f}, two new classes of stochastic processes: $\sum(H)$ and $\sum_{s}(H)$. The motivation came from the study of processes of the form:
\begin{equation}\label{eq0}
	X=N+A
\end{equation}
where $A$ is a non-decreasing and continuous stochastic process such that $dA_{t}$ is carried by the set of zeros of some stochastic process. The equation \eqref{eq0} has played a capital role in many probabilistic studies. For instance: the family of Azéma-Yor martingales, the resolution of Skorokhod embedding problem, the Skorokhod reflection equation, the study of Brownian local times and the study of zeros of continuous martingales \cite{1}. A large class containing several stochastic processes satisfying \eqref{eq0} is the class $\sum$. The class $\sum$ has been introduced by Yor \cite{y1} and studied in a series of articles \cite{nik}, \cite{pat}, \cite{mult}, \cite{naj}, \cite{naj1}, \cite{naj2} and \cite{naj3}. 

The interest of the present work lies in results established in \cite{mult} and \cite{pat}. In  \cite{mult}, the multiplicative decomposition formulas have been given for submartingales of $\sum$ class. In \cite{pat}, some representation results allowing to recover any process of $\sum$ class from its final value and the last time it visited the origin have been established. More precisely, these formulas are of the form:
\begin{equation}\label{rep}
	X_{t}=\E[X_{\infty}1_{\{L\leq t\}}|\mathcal{F}_{t}]
\end{equation}
where, $X$ is a process of $\sum$, $X_{\infty}=\lim_{t\to\infty}{X_{t}}$ and $L=\sup\{t\geq0; X_{t}=0\}$. These kind of results are useful. For instance, from these results, a general framework to study last passage times, suprema and drawdowns of class $\sum$ is given in \cite{pat}. Theses results have also played an important role in \cite{naj}, \cite{naj1}, \cite{naj2} and \cite{naj3} to establish a remarkable class of $\sigma-$ finite measures.

Inspired by these works, the aim of this paper is to establish multiplicative decomposition and representation  results for  $\sum(H)$ and $\sum_{s}(H)$  classes.

The paper is organized as follows: in Section 2, we shall give a brief summary of stochastic calculus for signed measures and definitions of classes $\sum$, $\sum(H)$ and $\sum_{s}(H)$. In Section 3, some results on multiplicative decomposition of  $\sum(H)$ and $\sum_{s}(H)$  classes are provided. Finally, in Section 4, we establish representation results allowing us to write the processes of $\sum(H)$ and $\sum_{s}(H)$ classes as in Equation \eqref{rep}. }

\section{Stochastic calculus for signed measures}

{ \fontfamily{times}\selectfont
 \noindent 
 We start by giving some notations that will be used throughout this paper. Consider a measure space $(\Omega, \mathcal{F}_{\infty}, \Qv)$, where $\Qv$ is a bounded signed measure. Let $\Pv$ be a probability measure on $\mathcal{F}_{\infty}$ such that $\Qv<<\Pv$. We shall always use the following notations:
\begin{itemize}
	\item $D_{t}=\frac{d\Qv|_{\mathcal{F}_{t}}}{d\Pv|_{\mathcal{F}_{t}}}$ where $\mathcal{F}$ is a right continuous filtration completed with respect to $\Pv$ such that $\mathcal{F}_{\infty}=\vee_{t}\mathcal{F}_{t}$. We shall consider that $D$ is continuous in this paper. Note also that $D$ is a uniformly integrable martingale (see Beghdadi-Sakrani \cite{sak}).
	\item $H=\{t: D_{t}=0\}$.
	\item $g=\sup{H}$. 
	\item $\overline{g}=0\vee g$. 
	\item $\gamma_{t}=0\vee\sup\{s\leq t, D_{s}=0\}$.
	\item If $X$ is an adapted process with respect to $(\mathcal{F}_{t})$, we shall denote $\widetilde{X}:=X_{.+\overline{g}}$.
	\item Throughout this paper, $\E$ denotes the expectation with respect to the probability measure $\Pv$.
\end{itemize}

Note that $g$ and $\overline{g}$ are not stopping times with respect to the filtration $(\mathcal{F}_{t})_{t\geq0}$. Then, we shall denote $(\mathcal{F}^{g}_{t})$, the smallest right continuous filtration containing $(\mathcal{F}_{t})$ for which $\overline{g}$ and $g$ are stopping times. Hence, the filtration $(\mathcal{F}^{g}_{\overline{g}+t})$ is well defined and will be denoted $(\mathcal{F}_{t+g})$.

The class of stochastic processes $\sum$ plays an important role in this work. We recall its definition in what follows.
\begin{defn}
We say that a stochastic process $X$ is of class $\sum$ if it decomposes as $X=N+A$, where
\begin{enumerate}
	\item $N$ is a càdlàg local martingale,
	\item $A$ is an adapted continuous finite variation process starting at 0,
	\item $\int_{0}^{t}{1_{\{X_{u}\neq0\}}dA_{u}}=0$ for all $t\geq0$.
\end{enumerate}
\end{defn}
\begin{defn}
A stochastic process $X$ is said of class $\Dv$ if \\$\{X_{\tau}:$ $\tau<\infty$ $is$ $a$ $stopping$ $time\}$ is uniformly integrable.
\end{defn}

Now, we shall recall some results of stochastic calculus for signed measures. P.A. Meyer suggested to use signed measures in order to generalize the Paul Lévy's Theorem. This generalization was established by Ruiz de Chavez in $\cite{chav}$ and completed by Beghdadi-Sakrani in \cite{sak}. The basis of stochastic calculus for signed measure theory have been established in both papers cited above. Note that the authors of \cite{chav} and \cite{sak} did not use the same definition of martingale with respect to a signed measure. A martingale with respect to a signed measure was defined in \cite{chav} as follows.
\begin{defn}\label{d1}
We consider a measure space $(\Omega,\mathcal{F}_{\infty},\Qv)$, where $\Qv$ is a bounded signed measure.
Let $\Pv$ be a probability measure on $\mathcal{F}_{\infty}$ such that $\Qv<<\Pv$. $(\mathcal{F}_{t})_{t\geq0}$ is a right continuous filtration, completed with respect to $\Pv$ such that $\mathcal{F}_{\infty}=\vee_{t}\mathcal{F}_{t}$ and $D_{t}=\frac{d\Qv|_{\mathcal{F}_{t}}}{d\Pv|_{\mathcal{F}_{t}}}$. We say that a $(\mathcal{F}_{t})_{t\geq0}-$  adapted process $X$ is a $(\Qv,\Pv)-$ martingale if:
\begin{enumerate}
	\item $X$ is a $\Pv-$ semi martingale.
	\item $XD$ is a $\Pv-$ martingale.
\end{enumerate}
$X$ is said $(\Qv,\Pv)-$ local martingale (or uniformly integrable $(\Qv,\Pv)-$ martingale) if $DX$ is a $\Pv-$ local martingale (or uniformly integrable $\Pv-$ martingale ).
\end{defn}
 
Note that it is only  due to assumption 1) that we can use usual stochastic calculus on the class of processes defined above. The class of stochastic processes $\sum(H)$ was defined from Definition \ref{d1} as follows:
\begin{defn}
Let $X$ be a non-negative $\Pv-$ semimartingale which decomposes as:
$$X_{t}=M_{t}+A_{t}.$$
We say that $X$ is of $\sum(H)$ class if:
\begin{enumerate}
	\item $M$ is a càdlàg $(\Qv,\Pv)-$ local martingale, with $M_{0}=0$;
	\item $A$ is a continuous nondecreasing process, with $A_{0}=0$;
	\item the measure $(dA_{t})$ is carried by the set $\{t: X_{t}=0\}\cup H$.
\end{enumerate}
\end{defn}

Now, we recall the definition of martingale with respect to a signed measure used in \cite{sak}. In the next definition, we shall take $\Pv =|\Qv|$, where $\Qv$ is a signed measure such that $|\Qv|(\Omega)=1$.
\begin{defn} 
Let $X$ be an adapted process with respect to the filtration $\mathcal{F}$.
\begin{enumerate}
	\item $X$ is called a $\Qv-$ martingale if: $\E[|X_{t}|]<+\infty$, $\forall t\geq0$ and $\Qv(X_{T})=\Qv(X_{0})$ for any bounded stopping time $T$.
	\item $X$ is called uniformly integrable $\Qv-$ martingale if $XD$ is  a $\Pv-$  martingale which is uniformly integrable .
	\item $X$ is a $\Qv-$  local martingale if $DX$ is a $\Pv-$  local martingale.
\end{enumerate} 
\end{defn}

Remark that we can not apply stochastic calculus on $\Qv-$   martingales, because a $\Qv-$   martingale is not necessarily a $\Pv-$ semimartingale. But, there exists a relation between the two above definitions of a martingale with respect to a signed measure. It is given in the next proposition.
\begin{prop}\label{compa}
Let $\Pv^{'}$ be a probability measure on $\mathcal{F}_{\infty}$ such that $\Qv<<\Pv^{'}$. If, $X$ is a $(\Qv,\Pv^{'})-$ martingale, then $X$ is a $\Qv-$ martingale.
\end{prop}
\begin{proof}
See Proposition 1.1 of \cite{sak}. 
\end{proof}

The theory of stochastic calculus for signed measures which allows us to deal with $\Qv-$ martingales was established in \cite{sak}. Now, we recall some results of this theory. We begin by quoting the following proposition of  Azema and Yor $\cite{1}$.
\begin{prop}\label{14}
Let $(V_{t})_{t\geq0}$ be a $(\mathcal{F}_{g+t})_{t\geq0}-$   optional process. There exists a unique $(\mathcal{F}_{t})_{t\geq0}-$   optional process $(U_{t})_{t\geq0}$  which vanishes on $H$ such that $\forall t\geq0$, $U_{\overline{g}+t}=V_{t}$ and $U_{0}=V_{0}$ on $\{\overline{g}=0\}$. That defines a function
$\rho:V\longmapsto U$.\\
$\rho$ is linear, nonnegative and preserves products.
\end{prop}
\begin{proof}
(See Azema and Yor \cite{1})
\end{proof}

Now, we shall recall the definition of stochastic integral with respect to a signed measure $\Qv$.
\begin{defn}
Let $X$ be a uniformly integrable $\Qv-$   martingale (resp. finite variation under $\Qv$) and h, a progessive process such that:\\
$\int^{t}_{0}h^{2}_{\overline{g}+s}d\langle X_{\overline{g}+.}\rangle_{s}<+\infty$, $\forall t\geq0$ (resp. $\int^{t}_{0}|h_{\overline{g}+s}||d X_{\overline{g}+s}|<+\infty$). We define the stochastic integral of $h$ with respect to $X$ under the signed measure $\Qv $ by:
\begin{equation}%$\ref{int}$
_\Qv\int^{t}_{0}h_{s}dX_{s}=\rho\left(\int_{0}^{.}h_{\overline{g}+s}d\widetilde{X}_{s}\right)_{t}.
\end{equation}
\end{defn}

\begin{rem}
A process $X$ is said nondecreasing (resp. finite variation) under $\Qv$ when the process $\widetilde{X}_{\cdot}=X_{\cdot+\overline{g}}$ is non-decreasing (resp. finite variation) under $\Pv$.
\end{rem}

We note $[X]^{\Qv}=\rho([\widetilde{X}])_{.}$.
\begin{theorem}
If $X$ is a uniformly integrable $\Qv-$ martingale, then $[X]^{\Qv}$ is the unique process adapted, right continuous, nondecreasing on $[\overline{g}, +\infty[$ and null on $H$ such that $X^{2}-[X]^{\Qv}$ is a $\Qv-$   local martingale.
\end{theorem}
\begin{proof}
(See S.Beghdadi-Sakrani \cite{sak})
\end{proof}

Now, we recall the Itô's Theorem for signed measures. 
\begin{theorem}
Let $X:=(X^{1},\ldots,X^{d})$ be a vector of $d$ right continuous $\Qv-$  semimartingales and $F\in$ $\mathcal{C}^{2}(\R^{d},\R)$. Then $F(X)$ is a right continuous $\Qv-$  semimartingale and for any $t\geq0$,
\begin{equation}
F(X_{t})=F(X_{\gamma_{t}})+\sum_{i}{_\Qv\int^{t}_{0}}\frac{\partial F}{\partial x_{i}}(X_{s})dX^{i}_{s}+\frac{1}{2}\sum_{i,j}{_\Qv\int^{t}_{0}}\frac{\partial^{2} F}{\partial x_{i}\partial x_{j}}(X_{s})d[X^{i},X^{j}]^{\Qv}_{s}
\end{equation}
\end{theorem}
\begin{proof}
(See Beghdadi-Sakrani \cite{sak})
\end{proof}

In what follows, we recall Tanaka's formulas for signed measures that we established in \cite{f}. 
\begin{theorem}\label{ito}
Let X be a right continuous $\Qv-$ semimartingal. For all real $a$, we have:
\begin{equation} \label{eq1}
|X_{t}-a|=|X_{\gamma_{t}}-a|+_{\Qv}\int^{t}_{0}{\rm sgn}(X_{s}-a)dX_{s}+_{\Qv}L_{t}^{a}(X)
\end{equation}

\begin{equation} \label{eq2}
(X_{t}-a)^{+}=(X_{\gamma_{t}}-a)^{+}+_{\Qv}\int^{t}_{0}1_{\{X_{s}>a\}}dX_{s}+\frac{1}{2}{_{\Qv}L_{t}^{a}(X)}
\end{equation}

\begin{equation} \label{eq3}
(X_{t}-a)^{-}=(X_{\gamma_{t}}-a)^{-}-_{\Qv}\int^{t}_{0}1_{\{X_{s}\leq a\}}dX_{s}+\frac{1}{2}{_{\Qv}L_{t}^{a}(X)}
\end{equation}
where ${_{\Qv}L_{t}^{a}(X)}=\rho \big(L_{.}^{a}(X_{.+\overline{g}})\big)_{t}$ and $(L_{t}^{a}(X_{.+\overline{g}}))_{t\geq0}$ is the classical semimartingale local time of $(X_{t+\overline{g}})_{t\geq0}$.
\end{theorem}
\begin{proof}
(See  \cite{f})
\end{proof}

Now, we recall the definition of class of stochastic processes $\sum_{s}(H)$.
\begin{defn}
Let $(X_{t})_{t\geq0}$ be a non-negative process which decomposes as:
$$X_{t}=N_{t}+A_{t}.$$
We say that $(X_{t})_{t\geq0}$ is a stochastic process of the $\sum_{s}(H)$ class if:
\begin{enumerate}
	\item $(N_{t})_{t\geq0}$ is a càdlàg, uniformly integrable $\Qv-$ martingale vanishing on $H$;
	\item $(A_{t})_{t\geq0}$ is a process which is continuous on $]\overline{g},\infty[$ and null on $H$ such that: $\widetilde{A}_{t}=A_{\overline{g}+t}$ is non-decreasing;
	\item the measure $(d\widetilde{A}_{t})$ is carried by the set $\{t: \widetilde{X}_{t}=0\}$.
\end{enumerate}
\end{defn}

We conclude this section with the following results which constitute our little contribution in stochastic calculus for signed measures.
\begin{theorem}\label{0.1}
Let $K$ be a continuous, uniformly integrable $(\Qv,\Pv)-$ martingale  with respect to $(\mathcal{F}_{t})_{t\geq0}$ which vanishes on $H$ and at zero. Let $(L_{t})_{t\geq0}$ be the local time of $K$ at zero under $\Pv$. And let $({_{\Qv}}L_{t})_{t\geq0}$ be the local time of $K$ at zero under $\Qv$ (i.e ${_{\Qv}}L_{t}$ is the process defined in Theorem \ref{ito}). Then, the process,  $({_{\Qv}}L_{t}-L_{t})_{t\geq0}$ is a uniformly integrable $(\Qv,(\mathcal{F}_{t})_{t\geq0})$- martingale. More precisely for all $t\geq0$, 
\begin{align}
	{_{\Qv}}L_{t}-L_{t}=I_{\gamma_{t}}.
\end{align}
where $I_{t}=\int_{0}^{t}{{\rm sgn}(K_{s})dK_{s}}$,
\end{theorem}
\begin{proof}
Since $K$ is a continuous $(\Pv,(\mathcal{F}_{t})_{t\geq0})-$ semimartingale vanishing at zero, we obtain from Tanaka formula:
$$|K_{t}|=\int_{0}^{t}{{\rm sgn}(K_{s})dK_{s}}+L_{t}.$$
Furthermore, from Proposition \ref{compa}, $K$ is also a uniformly integrable $\Qv-$ martingale. Hence,  according to Theorem 3.4 of \cite{f} one gets:
$$|K_{t}|=|K_{\gamma_{t}}|+{_\Qv}\int_{0}^{t}{{\rm sgn}(K_{s})dK_{s}}+{_{\Qv}}L_{t}.$$
But $|K_{\gamma_{t}}|=0$ since $K$ vanishes on $H$. Then, it follows that:
$$|K_{t}|={_\Qv}\int_{0}^{t}{{\rm sgn}(K_{s})dK_{s}}+{_{\Qv}}L_{t}.$$
Consequently, the following identity holds:
$$\int_{0}^{t}{{\rm sgn}(K_{s})dK_{s}}+L_{t}={_\Qv}\int_{0}^{t}{{\rm sgn}(K_{s})dK_{s}}+{_{\Qv}}L_{t}.$$
Therefore, we obtain: 
$${_{\Qv}}L_{t}-L_{t}=\int_{0}^{t}{{\rm sgn}(K_{s})dK_{s}}-{_\Qv}\int_{0}^{t}{{\rm sgn}(K_{s})dK_{s}}.$$
But, thanks to Proposition 2 of \cite{chav}, $\int_{0}^{\cdot}{{\rm sgn}(K_{s})dK_{s}}$ is a uniformly integrable $(\Qv,\Pv)-$ martingale with respect to $(\mathcal{F}_{t})_{t\geq0}$. Then, it is also a uniformly integrable $\Qv-$ martingale with respect to $(\mathcal{F}_{t})_{t\geq0}$. And from Proposition 2.3 of \cite{sak}, ${_\Qv}\int_{0}^{\cdot}{{\rm sgn}(K_{s})dK_{s}}$ is a uniformly integrable $(\Qv,(\mathcal{F}_{t})_{t\geq0})-$ martingale. Consequently, $({_{\Qv}}L_{t}-L_{t})_{t\geq0}$ is a uniformly integrable $(\Qv,(\mathcal{F}_{t})_{t\geq0})-$ martingale. Since $K$ is, at the same time, a $\Pv-$ semimartingale and a uniformly integrable $\Qv-$ martingale, it follows from Proposition 2.3 of \cite{sak} that the stochastic process $(I_{t})_{t\geq0}$ is well defined  and that
$${_\Qv}\int_{0}^{t}{{\rm sgn}(K_{s})dK_{s}}=I_{t}-I_{\gamma_{t}}.$$
Then, we obtain that,
$${_{\Qv}}L_{t}-L_{t}=I_{\gamma_{t}}.$$
This completes the proof.

\end{proof}

\begin{lem}\label{lem}
Let $(M_{t})_{t\geq0}$ be a non-negative $(\Qv,(\mathcal{F}_{t})_{t\geq0})-$ martingale which is uniformly integrable, with $M_{t}=1$, $\forall t\in H$. And let $(C_{t})_{t\geq0}$ be  a right continuous and adapted  process which is non-decreasing under $\Qv$ such that: $\forall t\in H$, $C_{t}=1$. Then, we have:
\begin{align}
	\E(M_{\infty}C_{\infty})=1+\E\left(\int_{\overline{g}}^{+\infty}{M_{s}dC_{s}}\right).
\end{align}
\end{lem}
\begin{proof}
By definition, $DM$ is a uniformly integrable $(\Pv,(\mathcal{F}_{t})_{t\geq0})-$ martingale which vanishes on $H$. Hence, thanks to  \cite[Theorem 3.2]{1} (Quotient Theorem) , it follows that: $\widetilde{M}=M_{\cdot+\overline{g}}$ is a uniformly integrable $(\Pv,(\mathcal{F}_{t+\overline{g}})_{t\geq0})-$   martingale, with $\widetilde{M}_{0}=1$. Furthermore, we have $\widetilde{C}=C_{\cdot+\overline{g}}$ is a continuous and non-decreasing process which is  adapted with respect to $(\mathcal{F}_{t+\overline{g}})_{t\geq0}$, with, $\widetilde{C}_{0}=1$. Then, according to  Lemma 2.2 of \cite{mult}, one gets that:
$$\E(\widetilde{M}_{\infty}\widetilde{C}_{\infty})=1+\E\left(\int_{0}^{+\infty}{\widetilde{M}_{s}d\widetilde{C}_{s}}\right),$$ 
$$\hspace{2.25cm}=1+\E\left(\int_{\overline{g}}^{+\infty}{M_{s}dC_{s}}\right).$$
But $\widetilde{M}_{\infty}=M_{\infty}$ and $\widetilde{C}_{\infty}=C_{\infty}$. Consequently, we obtain:
$$\E(M_{\infty}C_{\infty})=1+\E\left(\int_{\overline{g}}^{+\infty}{M_{s}dC_{s}}\right).$$
This ends the proof.

\end{proof}

\begin{theorem}\label{qt}
Let $K$ be a continuous uniformly integrable $(\Qv,(\mathcal{F}_{t})_{t\geq0})-$ martingale such that $K_{\overline{g}}=0$. Let us take $\overline{K}_{t}=\sup_{\gamma_{t}\leq s\leq t}{K_{s}}$ and $X_{t}=\overline{K}_{t}-K_{t}$. Then, for any locally bounded borel function $f$, the process
$$f(\overline{K}_{t})X_{t}-{_{\Qv}}\int_{0}^{t}{f(\overline{K}_{s})d\overline{K}_{s}}$$
is $(\Qv,(\mathcal{F}_{t})_{t\geq0})-$ local martingale.
\end{theorem}
\begin{proof}
Let $\Pv^{'}=\frac{|D_{\infty}|}{\E|D_{\infty}|}\Pv$. Since $K$ is a continuous uniformly integrable $(\Qv,(\mathcal{F}_{t})_{t\geq0})-$ martingale, it follows that $DX$ is a continuous uniformly integrable $(\Pv,(\mathcal{F}_{t})_{t\geq0})-$ martingale null on $H=\{t\geq0; |D_{t}|=0\}$. Hence, we can apply the Quotient Theorem ( \cite[Theorem 3.2]{1}) on process $DX$ (between $\Pv$ and $\Pv^{'}$) and to obtain that 
$$\frac{D_{t+\overline{g}}K_{t+\overline{g}}}{|D_{t+\overline{g}}|}={\rm sgn}(D_{t+\overline{g}})K_{t+\overline{g}},\hspace{0,5cm}for\hspace{0,1cm} all\hspace{0,1cm} t>0$$
 is a continuous uniformly integrable $(\Pv^{'},(\mathcal{F}_{t+g})_{t\geq0})-$ martingale. But $D$ is a continuous process and for all $t>0$, $D_{t+\overline{g}}\neq0$. Hence, ${\rm sgn}(D_{t+\overline{g}})=cste$. Consequently,  $K_{t+\overline{g}}$ is a uniformly integrable $(\Pv^{'},(\mathcal{F}_{t+g})_{t\geq0})$ - martingale. Furthermore, we have: 
$$\overline{K}_{t+\overline{g}}=\sup_{\overline{g}\leq s\leq t+\overline{g}}{K_{s}}=\sup_{s\leq t}{K_{s+\overline{g}}},$$ and $\overline{K}_{0}=K_{\overline{g}}=0$. Then, $X_{\cdot+\overline{g}}$ is a submartingale of class $\sum$. Therefore, for any locally bounded borel function $f$, the process,
$$f(\overline{K}_{t+\overline{g}})X_{t+\overline{g}}-\int_{0}^{t}{f(\overline{K}_{s+\overline{g}})d\overline{K}_{s+\overline{g}}}$$
is a $(\Pv^{'},(\mathcal{F}_{t+g})_{t\geq0})-$ local martingale. This implies that:
$$\rho(f(\overline{K}_{\cdot+\overline{g}})X_{\cdot+\overline{g}})_{t}-{_{\Qv}}\int_{0}^{t}{f(\overline{K}_{s})d\overline{K}_{s}},$$
is a $(\Qv,(\mathcal{F}_{t})_{t\geq0})-$ locale martingale. Remark furthermore that for all $t\in H$, $$X_{t}=\overline{K}_{t}-K_{t}=K_{t}-K_{t}=0.$$ 
Consequently, 
$$\rho(f(\overline{K}_{\cdot+\overline{g}})X_{\cdot+\overline{g}})_{t}=f(\overline{K}_{t})X_{t}.$$
This completes the proof.

\end{proof}
 }
\section{Multiplicative decompositions on classes \texorpdfstring{$\sum(H)$}{} and \texorpdfstring{$\sum_{s}(H)$}{}}

In this section, we shall give some formulas of multiplicative decomposition on stochastic processes of  $\sum(H)$ and $\sum_{s}(H)$  classes. Beforehand, we give the following proposition.

\begin{prop}\label{sig}
Let $X=N+A$ be a continuous process of $\sum(H)$ class. Then, $|XD|$ is a submartingale of class $\sum$.
\end{prop}
\begin{proof}
Firstly, note that $D$ is a continuous $\Pv-$ martingale. This implies that $DX$ is a continuous $\Pv$- semimartingale since $X$ is also a $\Pv$- semimartingale. Then, thanks to Tanaka formula, it follows that,
$$|X_{t}D_{t}|=\int_{0}^{t}{{\rm sgn}(D_{s}X_{s})d(D_{s}X_{s})}+L^{0}_{t},$$ 
$L^{0}_{t}$ being the  local time of the semimartingale $DX$ at zero. But we can write $D_{t}X_{t}$ as follows
$$D_{t}X_{t}=D_{t}N_{t}+\int_{0}^{t}{A_{s}dD_{s}}+\int_{0}^{t}{D_{s}dA_{s}}.$$
One gets that
$$|X_{t}D_{t}|=\int_{0}^{t}{{\rm sgn}(D_{s}X_{s})d(D_{s}N_{s})}+\int_{0}^{t}{{\rm sgn}(D_{s}X_{s})A_{s}dD_{s}}+\int_{0}^{t}{{\rm sgn}(D_{s}X_{s})D_{s}dA_{s}}+L^{0}_{t}.$$
But $dA$ is carried by  $H\cup\{t; X_{t}=0\}=\{t; D_{t}X_{t}=0\}$. This allows us to see that $\int_{0}^{t}{{\rm sgn}(D_{s}X_{s})D_{s}dA_{s}}=0$ since ${\rm sgn}(D_{s}X_{s})=0$ on $\{t; D_{t}X_{t}=0\}$. So, we obtain
$$|X_{t}D_{t}|=\int_{0}^{t}{{\rm sgn}(D_{s}X_{s})d(D_{s}N_{s})}+\int_{0}^{t}{{\rm sgn}(D_{s}X_{s})A_{s}dD_{s}}+L^{0}_{t}.$$
By definition, $(L^{0}_{t})$ is a continuous and non-decreasing process which vanishes at zero and  $dL^{0}_{t}$ is carried by $\{t; D_{t}X_{t}=0\}$. Note also that
$\int_{0}^{t}{{\rm sgn}(D_{s}X_{s})d(D_{s}N_{s})}$ is a $\Pv-$ local martingale since $DN$ is also a $\Pv-$ local  martingale. It is the same for $\int_{0}^{t}{{\rm sgn}(D_{s}X_{s})A_{s}dD_{s}}$; which completes the proof.

\end{proof}

A direct consequence of the above Proposition gives us the multiplicative decomposition of the process $DX$, $X$ being a stochastic process of the $\sum(H)$ class.
\begin{coro}
Let $X=N+A$ be a continuous process of $\sum(H)$ class. Then, there exists a unique continuous $\Pv-$ local martingale $M$ which is strictly  positive  such that $M_{0}=1$ and
$$|D_{t}|X_{t}=\frac{M_{t}}{I_{t}}-1,$$
where $I_{t}=\inf_{s\leq t}{M_{s}}$. Precisely the martingale $M$ is given by: 
$$M_{t}=(1+|D_{t}|X_{t})\exp(-L^{0}_{t}),$$ 
$L^{0}_{t}$ being the local time of $DX$ at zero.
\end{coro}
\begin{proof}
It follows from Proposition \ref{sig} that $|D|X=|DX|$ is a submartingale of class $\sum$. It suffices to apply \cite[Proposition 2.4]{mult}
to conclude.

\end{proof}

Now, we shall give the main results of this section on the $\sum(H)$ class.

\subsection{Multiplicative decomposition in \texorpdfstring{$\sum(H)$}{sigma(H)}}
\begin{prop}\label{dec1}
Let $M$ be a positive, continuous $(\Qv,\Pv)$- local martingale with $M_{0}=1$. Then,
$$X_{t}=\frac{M_{t}}{I_{t}}-1$$ is a process of $\sum(H)$ class, with  $I_{t}=\inf_{s\leq t}{M_{s}}$.
\end{prop}
\begin{proof}
An application of Itô formula gives us:
$$X_{t}=\int_{0}^{t}{\frac{dM_{s}}{I_{s}}}+\int_{0}^{t}{M_{s}d\left(\frac{1}{I_{s}}\right)}.$$
But 
$$M_{s}=I_{s}+I_{s}X_{s}.$$
Then, we obtain that:
$$X_{t}=\int_{0}^{t}{\frac{dM_{s}}{I_{s}}}+\int_{0}^{t}{I_{s}d\left(\frac{1}{I_{s}}\right)}+\int_{0}^{t}{X_{s}I_{s}d\left(\frac{1}{I_{s}}\right)}.$$
Remark that we have also: 
$$\frac{X_{t}}{M_{t}}=\frac{1}{I_{t}}-\frac{1}{M_{t}}$$ 
and that 
$$\frac{1}{I_{t}}=\sup_{s\leq t}\left(\frac{1}{M_{s}}\right).$$
This means that $d\left(\frac{1}{I_{s}}\right)$ is carried by $\{t; X_{t}=0\}$. This implies that 
$$\int_{0}^{t}{X_{s}I_{s}d\left(\frac{1}{I_{s}}\right)}=0.$$
Then, it follows that:
$$X_{t}=\int_{0}^{t}{\frac{dM_{s}}{I_{s}}}+\int_{0}^{t}{I_{s}d\left(\frac{1}{I_{s}}\right)}.$$
Which can also be written as follows:
$$X_{t}=\int_{0}^{t}{\frac{dM_{s}}{I_{s}}}+\log\left(\frac{1}{I_{t}}\right).$$
Thanks to Proposition 2 of \cite{chav}, $\left(\int_{0}^{t}{\frac{dM_{s}}{I_{s}}}\right)_{t\geq0}$ is a $(\Qv,\Pv)-$ locale martingale. We have also that \\$\left(\log\left(\frac{1}{I_{t}}\right)\right)_{\geq0}$ is a continuous, non-decreasing process which vanishes at zero. Furthermore, $d\left(\log\left(\frac{1}{I_{t}}\right)\right)$ is carried by $H\cup\{t; X_{t}=0\}$ since $d\left(\frac{1}{I_{t}}\right)$ is carried by $\{t; X_{t}=0\}$.
Consequently, $X$ is a process of $\sum(H)$ class.

\end{proof}

\begin{coro}
Let $X=N+A$ be a continuous and nonnegative semimartingale which vanishes on $H$. Then, the following assertions are equivalent:
\begin{enumerate}
	\item $X$ is a process of $\sum(H)$ class;
	\item there exists a unique continuous and positive $(\Qv,\Pv)-$ local martingale $M$ such that $M_{0}=1$ and
	$$X_{t}=\frac{M_{t}}{I_{t}}-1,$$
	where $I_{t}=\inf_{s\leq t}{M_{s}}$. Precisely, the 	$(\Qv,\Pv)-$ local martingale $M$ is given by \\$M_{t}=(1+X_{t})\exp(-A_{t})$.
\end{enumerate}
\end{coro}
\begin{proof}
$(1)\Rightarrow(2)$ Thanks to Theorem 2.1 of \cite{f}, we have that for any locally bounded Borel function $f$, $$\left(f(A_{t})X_{t}-\int_{0}^{A_{t}}{f(x)dx}\right)_{t\geq0}$$ is a $(\Qv,\Pv)-$ local martingale. Take $f(x)=\exp{(-x)}$. It follows that, $(\exp(-A_{t})X_{t}+\exp(-A_{t})-1)_{t\geq0}$ is a $(\Qv,\Pv)-$ local martingale. This implies that $M=(\exp(-A_{t})X_{t}+\exp(-A_{t}))_{t\geq0}$ is a $(\Qv,\Pv)-$ local martingale. It is easy to see that $M$ is a positive process such that $M_{0}=1$. Now, we put $C_{t}=\exp(A_{t})$. It follows that 
$$M_{t}=\frac{(1+X_{t})}{C_{t}}.$$
This is equivalent to write that:
$$\frac{X_{t}}{M_{t}}=C_{t}-\frac{1}{M_{t}}.$$
Remark that $dC_{t}=C_{t}dA_{t}$ is carried by $H\cup\{t\geq0; X_{t}=0\}$. But $X$ vanishes on $H$. Then, $dC_{t}$ is carried by $\{t\geq0; X_{t}=0\}$. This means that $dC_{t}$ is carried by $\{t\geq0; \frac{X_{t}}{M_{t}}=0\}$, since $\forall t\geq0$, $M_{t}\neq0$. So, by applying Skorokhod's Lemma, we obtain that:
$$C_{t}=\sup_{s\leq t}{\left(\frac{1}{M_{s}}\right)}=\frac{1}{I_{t}}.$$
Then, 
$$\frac{X_{t}}{M_{t}}=\frac{1}{I_{t}}-\frac{1}{M_{t}}.$$ 
Consequently, 
$$X_{t}=\frac{M_{t}}{I_{t}}-1.$$
$(2)\Rightarrow(1)$ follows from a direct application of Proposition \ref{dec1}.

\end{proof}

Now, we shall give the main result of this section on the multiplicative decomposition of stochastic processes of the $\sum_{s}(H)$ class.
\subsection{Multiplicative decomposition in \texorpdfstring{$\sum_{s}(H)$}{sigma(s)(H)}}

\begin{theorem}\label{0.2}
Let $X$ be a continuous and non-negative stochastic process which is adapted with respect to the filtration $(\mathcal{F}_{t})_{t\geq0}$. Then, the following are equivalent:
\begin{enumerate}
	\item $X$ is a process of $\sum_{s}(H)$ class;
	\item there exists a positive uniformly integrable $(\Qv,(\mathcal{F}_{t})_{t\geq0})-$ martingale $M$, with $M_{t}=1$ $\forall t\in H$, such that
	
\begin{align}
	X_{t}=\frac{M_{t}}{J_{t}}-1,
\end{align}
where $J_{t}=\inf_{\gamma_{t}\leq s\leq t}{M_{s}}$.
\end{enumerate}
And this $(\Qv,(\mathcal{F}_{t})_{t\geq0})-$ martingale is given by $M_{t}=(1+X_{t})\exp(-A_{t})$.
\end{theorem}
\begin{proof}
$(1)\Rightarrow(2)$
Let $X=N+A$ be a process of $\sum_{s}(H)$ class. It follows from Proposition 3.2 of \cite{f} that: $X_{\cdot+\overline{g}}=N_{\cdot+\overline{g}}+A_{\cdot+\overline{g}}$ is a submartingale of $\sum$ class. Then, from \cite[Proposition 2.4]{mult}, we have:
$$X_{t+\overline{g}}=\frac{M^{'}_{t}}{I^{g}_{t}}-1,$$
where $M^{'}_{t}=(1+X_{t+\overline{g}})\exp(-A_{t+\overline{g}})$ and $I^{g}_{t}=\inf_{s\leq t}{M_{s}^{'}}$. Now,  let $Y_{t}=\frac{M_{t}}{J_{t}}-1$,
where $M_{t}=(1+X_{t})\exp(-A_{t})$ and  $J_{t}=\inf_{\gamma_{t}\leq s\leq t}{M_{s}}$. Thanks to Theorem 3.3 of \cite{f}, we have that for any locally bounded Borel function $f$, $f(A_{t})X_{t}-{_{\Qv}}\int_{0}^{t}f(A_{s})dA_{s}$ is a uniformly integrable $\Qv-$ martingale with respect to the filtration $(\mathcal{F}_{t})_{t\geq0}$. By taking $f(x)=\exp(-x)$, one gets that $\exp(-A_{t})X_{t}-{_{\Qv}}\int_{0}^{t}\exp(-A_{s})dA_{s}$ is a uniformly integrable $\Qv-$ martingale with respect to $(\mathcal{F}_{t})_{t\geq0}$. But from Proposition 2.3 of \cite{sak}, we have:
$${_{\Qv}}\int_{0}^{t}\exp(-A_{s})dA_{s}=\int_{\gamma_{t}}^{t}\exp(-A_{s})dA_{s}=\exp(-A_{\gamma_{t}})-\exp(-A_{t}).$$
Since $A$ vanishes on $H$, it follows that:
$${_{\Qv}}\int_{0}^{t}\exp(-A_{s})dA_{s}=1-\exp(-A_{t}).$$
This implies that:
$$\exp(-A_{t})X_{t}-{_{\Qv}}\int_{0}^{t}\exp(-A_{s})dA_{s}=(1+X_{t})\exp(-A_{t})-1=M_{t}-1.$$
Then, $M$ is a positive uniformly integrable $(\Qv,(\mathcal{F}_{t})_{t\geq0})-$ martingale, with $M_{t}=1$ for all $t\in H$. On the other hand, we have
$$Y_{t+\overline{g}}=\frac{M_{t+\overline{g}}}{J_{t+\overline{g}}}-1.$$
But $M_{t+\overline{g}}=M^{'}_{t}$ and $J_{t+\overline{g}}=\inf_{\overline{g}\leq s\leq t+\overline{g}}{M_{s}}=\inf_{s\leq t}{M_{s+\overline{g}}}$. Therefore, $J_{t+\overline{g}}=I^{g}_{t}$. Consequently, $X_{t+\overline{g}}=Y_{t+\overline{g}}$. This implies that $\rho(X_{\cdot+\overline{g}})_{t}=\rho(Y_{\cdot+\overline{g}})_{t}$. But by definition, $X$ vanishes on $H$. Then, $X_{t}=\rho(Y_{\cdot+\overline{g}})_{t}$. Furthermore, we have $\forall t\in H$, $J_{t}=M_{t}=1$. Hence, $Y$ also vanishes on $H$. Consequently, $\forall t\geq0$, $X_{t}=Y_{t}$, since $X$ and $Y$ are continuous. This means that $X_{t}=\frac{M_{t}}{J_{t}}-1$.\\
$(2)\Rightarrow(1)$
Now, we suppose that $X_{t}=\frac{M_{t}}{J_{t}}-1$. It follows, from Theorem 2.2 of \cite{sak}, that:
$$X_{t}={_{\Qv}}\int_{0}^{t}{\frac{dM_{s}}{J_{s}}}+{_{\Qv}}\int_{0}^{t}{M_{s}d\left(\frac{1}{J_{s}}\right)}.$$
Then, by using the fact that $M_{s}=J_{s}+X_{s}J_{s}$  in the last equality, we get:
$$X_{t}={_{\Qv}}\int_{0}^{t}{\frac{dM_{s}}{J_{s}}}+{_{\Qv}}\int_{0}^{t}{J_{s}d\left(\frac{1}{J_{s}}\right)}+{_{\Qv}}\int_{0}^{t}{X_{s}J_{s}d\left(\frac{1}{J_{s}}\right)}.$$
Thanks to Proposition 2.3 of \cite{sak}, we have:
$${_{\Qv}}\int_{0}^{t}{J_{s}d\left(\frac{1}{J_{s}}\right)}=\int_{\gamma_{t}}^{t}{J_{s}d\left(\frac{1}{J_{s}}\right)}=\log\left(\frac{1}{J_{t}}\right)-\log\left(\frac{1}{J_{\gamma_{t}}}\right).$$
Hence,
$${_{\Qv}}\int_{0}^{t}{J_{s}d\left(\frac{1}{J_{s}}\right)}= \log\left(\frac{1}{J_{t}}\right),$$
since $J_{\gamma_{t}}=M_{\gamma_{t}}=1$. But we also have that:
$${_{\Qv}}\int_{0}^{t}{X_{s}J_{s}d\left(\frac{1}{J_{s}}\right)}=\rho\left(\int_{0}^{\cdot}{X_{s+\overline{g}}J_{s+\overline{g}}d\left(\frac{1}{J_{s+\overline{g}}}\right)}\right)_{t}$$
and
$$X_{t+\overline{g}}=\frac{M_{t+\overline{g}}}{J_{t+\overline{g}}}-1=\frac{M_{t+\overline{g}}}{\inf_{s\leq t}{M_{s+\overline{g}}}}-1.$$
Then, it follows that:
$$\frac{X_{t+\overline{g}}}{M_{t+\overline{g}}}=\frac{1}{\inf_{s\leq t}{M_{s+\overline{g}}}}-\frac{1}{M_{t+\overline{g}}}$$
$$\hspace{1.25cm}=\sup_{s\leq t}{\left(\frac{1}{M_{s+\overline{g}}}\right)}-\frac{1}{M_{t+\overline{g}}}.$$
This means that, $d\left(\frac{1}{J_{s+\overline{g}}}\right)$ is carried by $\{s\geq0;X_{s+\overline{g}}=0\}$. Consequently,
$$\int_{0}^{t}{X_{s+\overline{g}}J_{s+\overline{g}}d\left(\frac{1}{J_{s+\overline{g}}}\right)}=0.$$
Then,
$${_{\Qv}}\int_{0}^{t}{X_{s}J_{s}d\left(\frac{1}{J_{s}}\right)}=0.$$
So, it follows that:
$$X_{t}={_{\Qv}}\int_{0}^{t}{\frac{dM_{s}}{J_{s}}}+ \log\left(\frac{1}{J_{t}}\right).$$
Remark that $\log\left(\frac{1}{J_{\cdot}}\right)$ is a non-decreasing process under $\Qv$ vanishing on $H$. Since $M$ is a uniformly integrable  $\Qv-$ martingale, it follows thanks to Proposition 2.3 of \cite{sak} that ${_{\Qv}}\int_{0}^{\cdot}{\frac{dM_{s}}{J_{s}}}$ is a uniformly integrable $\Qv-$ martingale with respect to $(\mathcal{F}_{t})_{t\geq0}$. Therefore, $X$ is a stochastic process of $\sum_{s}(H)$ class.

\end{proof}

In what follows, we give some corollaries of Theorem \ref{0.2}.
\begin{coro}\label{**}
Let $K$ be a continuous, uniformly integrable $(\Qv,(\mathcal{F}_{t})_{t\geq0})-$ martingale which vanishes on $H$. Then, there exists a uniformly integrable $(\Qv,(\mathcal{F}_{t})_{t\geq0})-$ martingale $M$ which is  positive, with $M_{t}=1$ for all $t\in H$, such that \\$\left({_{\Qv}}L_{t}-\log\left(\frac{1}{J_{t}}\right)\right)_{t\geq0}$ is a uniformly integrable $(\Qv,(\mathcal{F}_{t})_{t\geq0})-$ martingale. With,
$J_{t}=\inf_{\gamma_{t}\leq s\leq t}{M_{s}}.$ 
\end{coro}
\begin{proof}
From Proposition 3.2 of \cite{f}, we have,
$$|K_{t}|={_\Qv}\int_{0}^{t}{{\rm sgn}(K_{s})dK_{s}}+{_{\Qv}}L_{t}.$$
Then, $|K|$ is a process of $\sum_{s}(H)$ class. And from Theorem \ref{0.2}, it follows that:
$$|K_{t}|=\frac{M_{t}}{J_{t}}-1,$$
where, $M_{t}=(1+|K_{t}|)\exp(-{_{\Qv}}L_{t})$. But from Theorem 2.2 of \cite{sak}, it follows that:
$$|K_{t}|={_\Qv}\int_{0}^{t}{\frac{dM_{s}}{J_{s}}}+\log\left(\frac{1}{J_{t}}\right).$$
Hence,
$${_{\Qv}}L_{t}-\log\left(\frac{1}{J_{t}}\right)={_\Qv}\int_{0}^{t}{\frac{dM_{s}}{J_{s}}}-{_\Qv}\int_{0}^{t}{{\rm sgn}(K_{s})dK_{s}}.$$
Consequently, $\left({_{\Qv}}L_{t}-\log\left(\frac{1}{J_{t}}\right)\right)_{t\geq0}$ is a uniformly integrable $(\Qv,(\mathcal{F}_{t})_{t\geq0})-$ martingale. The proof is now complete.

\end{proof}

\begin{coro}\label{***}
Let $K$ be a continuous, uniformly integrable $(\Qv,(\mathcal{F}_{t})_{t\geq0})-$ martingale which vanishes on $H$. Then, there exists a  positive, uniformly integrable $(\Qv,(\mathcal{F}_{t})_{t\geq0})-$ martingale $M$, with $M_{t}=1$ for all $t\in H$ such that:
\begin{align}
	{_{\Qv}}L_{t}=\log\left(\frac{1}{J_{t}}\right).
\end{align}
Where, $J_{t}=\inf_{\gamma_{t}\leq s\leq t}{M_{s}}$.
\end{coro}
\begin{proof}
Thanks to Corollary \ref{**}, there exists an $(\mathcal{F}_{t})_{t\geq0}-$ adapted, continuous, uniformly integrable $\Qv-$ martingale $M$ such that $M_{t}=1$ for all $t\in H$ and for which the process $$\left({_{\Qv}}L_{t}-\log\left(\frac{1}{J_{t}}\right)\right)_{t\geq0}$$ is a uniformly integrable $(\Qv,(\mathcal{F}_{t})_{t\geq0})-$ martingale. This means that $$\left(D_{t}\left({_{\Qv}}L_{t}-\log\left(\frac{1}{J_{t}}\right)\right)\right)_{t\geq0}$$ is an $(\mathcal{F}_{t})_{t\geq0}-$ adapted, uniformly integrable $\Pv-$ martingale. By an application of Quotient Theorem, it follows that $\left({_{\Qv}}L_{t+\overline{g}}-\log\left(\frac{1}{J_{t+\overline{g}}}\right)\right)_{t\geq0}$ is a uniformly integrable $(\Pv,(\mathcal{F}_{t+g})_{t\geq0})-$ martingale. But
$${_{\Qv}}L_{t+\overline{g}}-\log\left(\frac{1}{J_{t+\overline{g}}}\right)=L_{t}-\log\left(\frac{1}{\inf_{\overline{g}\leq s\leq t+\overline{g}}{M_{s}}}\right),$$
$L$ being the local time of the $\Pv-$ martingale $(K_{t+\overline{g}})_{t\geq0}$ which is $(\mathcal{F}_{t+g})_{t\geq0}-$ adapted. Remark that $\left(L_{t}-\log\left(\frac{1}{\inf_{\overline{g}\leq s\leq t+\overline{g}}{M_{s}}}\right)\right)_{t\geq0}$ is also a continuous process with finite variations which vanishes at zero. Consequently,
$$L_{t}-\log\left(\frac{1}{\inf_{\overline{g}\leq s\leq t+\overline{g}}{M_{s}}}\right)=0.$$
This means that:
$$L_{t}=\log\left(\frac{1}{\inf_{\overline{g}\leq s\leq t+\overline{g}}{M_{s}}}\right).$$
Then, it follows that:
$$\rho(L_{\cdot})_{t}=\rho\left(\log\left(\frac{1}{\inf_{\overline{g}\leq s\leq \cdot+\overline{g}}{M_{s}}}\right)\right)_{t}.$$
Hence,
$${_{\Qv}}L_{t}=\rho\left(\log\left(\frac{1}{\inf_{\overline{g}\leq s\leq \cdot+\overline{g}}{M_{s}}}\right)\right)_{t}.$$
Since $J_{t}=M_{t}=1$ on $H$, $\log\left(\frac{1}{J_{t}}\right)=0$  for all $t\in H$. Consequently, $\forall t\geq0$,
$${_{\Qv}}L_{t}=\log\left(\frac{1}{J_{t}}\right).$$
This ends the proof.

\end{proof}

\begin{coro}
Let $K$ be a continuous uniformly integrable $(\Qv,\Pv)-$ martingale which vanishes on $H$ and is adapted to the filtration $(\mathcal{F}_{t})_{t\geq0}$. Then, there exists a uniformly integrable $(\Qv,(\mathcal{F}_{t})_{t\geq0})-$ martingale $M$ which is positive, with $M_{t}=1$ $\forall t\in H$, such that $\left(L_{t}-\log\left(\frac{1}{J_{t}}\right)\right)_{t\geq0}$ is a uniformly integrable  $(\Qv,(\mathcal{F}_{t})_{t\geq0})-$ martingale. Here, $L$ is the local time of $K$ at zero and $J_{t}=\inf_{\gamma_{t}\leq s\leq t}{M_{s}}$.
\end{coro}
\begin{proof}
Thanks to Theorem \ref{0.1}, $(L_{t}-{_{\Qv}}L_{t})_{t\geq0}$ is a uniformly integrable  $\Qv-$ martingale with respect to the filtration $(\mathcal{F}_{t})_{t\geq0}$. But, from Corollary \ref{***}, there exists a uniformly integrable $(\Qv,(\mathcal{F}_{t})_{t\geq0})-$ martingale $M$ which is positive such that $M_{t}=1$, $\forall t\in H$, and ${_{\Qv}}L_{t}=\log\left(\frac{1}{J_{t}}\right)$. Consequently, $\left(L_{t}-\log\left(\frac{1}{J_{t}}\right)\right)_{t\geq0}$ is a uniformly integrable $(\Qv,(\mathcal{F}_{t})_{t\geq0})-$ martingale.

\end{proof}

%%%%%%%%%%%%%%%%%%%%%%%%%%%%%%%%%%%%%%%%%%%%%%%%%%%%%%%%%%%%%%%%%%%%%%%%%%%%%%%%%%%%%%%%%%%%%%%%%%%%%%%%%%%%%%%%%%%%%%%%%%%%%%%%%%%%%%%%%%%%%%%%%%%%%%%%%%%%%%%%
\section{Representation results in \texorpdfstring{$\sum(H)$, $\sum_{s}(H)$}{sigma(H) and sigma(s)(H)}}

In this section, we establish some representation formulas of stochastic processes of the  $\sum(H)$ and $\sum_{s}(H)$  classes. These results are similar to the ones obtained for stochastic processes of the class $\sum$ in \cite{pat}. More precisely, we give some formulas that allow us to recover any stochastic process of $\sum(H)$ class (or $\sum_{s}(H)$ class) from its final value and the last time it visited the origin. In other words, the goal in this section is to show that
\begin{equation}
	X_{t}=\E[X_{\infty}1_{\{\Gamma\leq t\}}|\mathcal{F}_{t}]
\end{equation}
where $\Gamma=\sup{\{t\geq g; X_{t}=0\}}$ and $X$ is a stochastic process of $\sum(H)$ class or $\sum_{s}(H)$ class. It is difficult to obtain it directly because here, $X$ is not necessarily a submartingale with respect to the probability measure $\Pv$. The case where $X$ is taken in $\sum_{s}(H)$, $X$ is not even a semimartingale. Thus, to overcome this difficulty, we will use the following remark:
\begin{rem}
Let $\Gamma=\sup{\{t\geq g; X_{t}=0\}}$. One has
$$\Gamma=g+L$$
where $g=\sup{\{t\geq0; D_{t}=0\}}$ and $L=\sup{\{t\geq0; X_{t+g}=0\}}$.
\end{rem}

Note that throughout this section, we shall assume that $g<\infty$. Remark that this assumption implies that $\overline{g}=g$.

\subsection{Representation formulas of processes of class \texorpdfstring{$\sum(H)$}{sigma(H)}} 

The following lemma gives us a relation between the $\sum(H)$ class and the class $\sum$. This lemma allows us to prove the main results of this subsection.
\begin{lem}\label{l1}
Let $X$ be an $(\mathcal{F}_{t})-$ adapted process of $\sum(H)$ class which vanishes on $H$ such that $dA_{t+ g}$ is carried by $\{t:X_{t+ g}=0\}$. Then, $X_{\cdot+ g}$ is a submartingale of the class $\sum$ with respect to $(\mathcal{F}_{t+g})_{t\geq0}$.
\end{lem}
\begin{proof}
We have $X_{t+ g}=N_{t+ g}+A_{t+ g}$. Since $X$ vanishes on $H$, one has: $X_{ g}=0$. Hence, it follows that:
$$X_{t+ g}=X_{t+ g}-X_{ g}=(N_{t+ g}-N_{ g})+(A_{t+ g}-A_{ g}).$$
But $A^{'}_{\cdot}=A_{\cdot+ g}-A_{ g}$ is an increasing and nonnegative process with $A^{'}_{0}=0$. Furthermore, $dA^{'}_{t}=dA_{t+ g}$ is carried by $\{t:X_{t+ g}=0\}$. Since $N$ is a $(\Qv,\Pv)-$ local martingale, we obtain, thanks to Theorem 4.2.1 of \cite{1}, that $N_{\cdot+ g}$ is a $(\Pv,(\mathcal{F}_{t+g}))-$ local martingale. Thus, $N_{\cdot+ g}-N_{ g}$ is a $(\Pv,(\mathcal{F}_{t+g})_{t\geq0})-$ local martingale which vanishes at zero. Consequently, $X_{\cdot+ g}$ is a submartingale of $\sum$ class.

\end{proof}

Now, we shall give the main results of this subsection.
\begin{theorem}\label{t1}
Let $X$ be an $(\mathcal{F}_{t})_{t\geq0}-$ adapted process of class $\big(\sum(H)\big)$ which vanishes on $H$, such that $dA_{t+ g}$ is carried by $\{t:X_{t+ g}=0\}$. Let $f:\R\longrightarrow\R$ be a Borel function and $\Gamma=\sup\{t\geq g; X_{t}=0\}$. Assume that the following condition holds:
\begin{itemize}
	\item $\left(N_{t}-N_{\gamma_{t}}\right)_{t\geq0}$ is a uniformly integrable $\Qv-$ martingale.
\end{itemize}
Then, there exist random variables $X_{\infty}$, $N_{\infty}$, $A_{\infty}$ such that $X_{t+ g}\longrightarrow X_{\infty}$, \\$\left(N_{t+ g}-N_{ g}\right)\longrightarrow N_{\infty}$, $(A_{t+ g}-A_{ g})\longrightarrow A_{\infty}$ as $t$ goes to $\infty$ almost everywhere on $\{L<\infty\}$. Moreover, for all stopping time $T<\infty$:
\begin{equation}
f(A_{T}-A_{\gamma_{T}})X_{T}=\E[f(A_{\infty})X_{\infty}1_{\{\Gamma\leq T\}}|\mathcal{F}_{T}].	
\end{equation}
In particular, for all stopping time $T<\infty$:
\begin{equation}
	X_{T}=\E[X_{\infty}1_{\{\Gamma\leq T\}}|\mathcal{F}_{T}].
\end{equation}
\end{theorem}
\begin{proof}
Thanks to Lemma \ref{l1}, $X_{\cdot+ g}=(N_{\cdot+ g}-N_{ g})+(A_{\cdot+ g}-A_{ g})$ is a submartingale of $\sum$ class, with $(A_{\cdot+ g}-A_{ g})$ its nondecreasing part and $(N_{\cdot+ g}-N_{ g})$ its martingale part. Since $\left(N_{t}-N_{\gamma_{t}}\right)_{t\geq0}$ is a uniformly integrable $(\Qv,(\mathcal{F}_{t})_{t\geq0})-$ martingale, the process $\left(D_{t}(N_{t}-N_{\gamma_{t}})\right)_{t\geq0}$ is a uniformly integrable $(\Pv,(\mathcal{F}_{t})_{t\geq0})-$ martingale. Hence, proceeding in the same way as in the proof of Theorem \ref{qt}, we get that $N_{\cdot+ g}-N_{ g}$ is a uniformly integrable $(\mathcal{F}_{t+g})_{t\geq0}-$ martingale.
Therefore, according to Corollary 3.2 of \cite{pat}, one has:
$$f(A_{T+ g}-A_{ g})X_{T+ g}=\E[f(A_{\infty})X_{\infty}1_{\{L\leq T\}}|\mathcal{F}_{T+g}]$$
for every stopping time $T<\infty$. It follows that:
$$\rho(f(A_{\cdot+ g}-A_{ g})X_{\cdot+ g})_{T}=\rho(\E[f(A_{\infty})X_{\infty}1_{\{L\leq \cdot\}}|\mathcal{F}_{\cdot+g}])_{T}.$$
Let $Y_{t}=f(A_{t}-A_{\gamma_{t}})X_{t}$. The process $Y$ vanishes on $H$ because $X$ is null on $H$. Furthermore, $\forall$ $t\geq0$,
$$Y_{t+ g}=f(A_{t+ g}-A_{ g})X_{t+ g}.$$
Then,
$$\rho(f(A_{\cdot+ g}-A_{ g})X_{\cdot+ g})_{T}=f(A_{T}-A_{\gamma_{T}})X_{T}$$
for every stopping time $T$. Now, let $Z_{t}=\E[f(A_{\infty})X_{\infty}1_{\{ g+L\leq t\}}|\mathcal{F}_{t}]$. Remark that the process $Z$ also vanishes on $H$ and $\forall$ $t\geq0$,
$$Z_{t+ g}=\E[f(A_{\infty})X_{\infty}1_{\{L\leq t\}}|\mathcal{F}_{t+g}].$$
Hence, we obtain:
$$Z_{t}=\rho(\E[f(A_{\infty})X_{\infty}1_{\{L\leq \cdot\}}|\mathcal{F}_{\cdot+g}])_{t}.$$
Consequently, for every stopping time $T$,
$$f(A_{T}-A_{\gamma_{T}})X_{T}=\E[f(A_{\infty})X_{\infty}1_{\{ g+L\leq T\}}|\mathcal{F}_{T}].$$
In particular, when $f\equiv1$, we obtain that:
$$X_{T}=\E[X_{\infty}1_{\{ g+L\leq T\}}|\mathcal{F}_{T}].$$
This completes the proof.

\end{proof}

\begin{theorem}
Let $X=N+A$ be a process of class $\big(\sum(H)\big)$ which vanishes on $H$, such that $dA_{t+ g}$ is carried by $\{t:X_{t+ g}=0\}$. Let $f:\R\longrightarrow\R$ be a Borel function and $\Gamma=\sup\{t\geq g; X_{t}=0\}$. Then, the following assertions hold:
\begin{enumerate}
	\item If $X_{\cdot+ g}$ is of class $\Dv$, then there exist integrable random variables $X_{\infty}$, $N_{\infty}$, $A_{\infty}$ such that $X_{t+ g}\longrightarrow X_{\infty}$, $(N_{t+ g}-N_{ g})\longrightarrow N_{\infty}$, $(A_{t+ g}-A_{ g})\longrightarrow A_{\infty}$ as $t\rightarrow\infty$ almost surely as well as in $L^{1}$ and for every stopping time $T$,
\begin{equation}
f(A_{T}-A_{\gamma_{T}})X_{T}=\E[f(A_{\infty})X_{\infty}1_{\{\Gamma\leq T\}}|\mathcal{F}_{T}].	
\end{equation}
	\item If $q:\R\longrightarrow \R-\{0\}$ is a Borel function such that $q(A_{\cdot+ g}-A_{ g})X_{\cdot+ g}$ is of class $\Dv$, then there exist random variables $X_{\infty}$, $N_{\infty}$, $A_{\infty}$ such that $X_{t+ g}\longrightarrow X_{\infty}$, $(N_{t+ g}-N_{ g})\longrightarrow N_{\infty}$, $(A_{t+ g}-A_{ g})\longrightarrow A_{\infty}$ as $t\rightarrow\infty$ almost everywhere on $\{L<\infty\}$. Moreover, for all stopping time $T<\infty$,
	\begin{equation}
	f(A_{T}-A_{\gamma_{T}})X_{T}=\E[f(A_{\infty})X_{\infty}1_{\{\Gamma\leq T\}}|\mathcal{F}_{T}]. 
	\end{equation}
\end{enumerate}
In particular, in both cases,  one has that for all stopping time $T<\infty$,
\begin{equation}
X_{T}=\E[X_{\infty}1_{\{\Gamma\leq T\}}|\mathcal{F}_{T}].
\end{equation}
\end{theorem}
\begin{proof}
Since $X$ vanishes on $H$, we have $X_{ g}=0$. Then, we obtain that $X_{\cdot+ g}=(N_{\cdot+ g}-N_{ g})+(A_{\cdot+ g}-A_{ g})$. From Lemma \ref{l1}, $X_{\cdot+ g}$ is a process of class $\sum$. Then, according to Theorem 3.1 of \cite{pat}, the following hold:
\begin{enumerate}
	\item If $X_{\cdot+ g}$ is of class $\Dv$, then there exist integrable random variables $X_{\infty}$, $N_{\infty}$, $A_{\infty}$ such that      $X_{t+ g}\longrightarrow X_{\infty}$, $(N_{t+ g}-N_{ g})\longrightarrow N_{\infty}$, $(A_{t+ g}-A_{ g})\longrightarrow A_{\infty}$ almost surely as well as in $L^{1}$ and for every stopping time $T$,
$$f(A_{T+ g}-A_{ g})X_{T+ g}=\E[f(A_{\infty})X_{\infty}1_{\{L\leq T\}}|\mathcal{F}_{T+g}].$$
  \item If $q:\R\longrightarrow \R-\{0\}$ is a Borel function such that $q(A_{\cdot+ g}-A_{ g})X_{\cdot+ g}$ is of class $\Dv$, then there exist random variables $X_{\infty}$, $N_{\infty}$, $A_{\infty}$ such that $X_{t+ g}\longrightarrow X_{\infty}$, $(N_{t+ g}-N_{ g})\longrightarrow N_{\infty}$, $(A_{t+ g}-A_{ g})\longrightarrow A_{\infty}$ almost everywhere on $\{L<\infty\}$ where $L=\sup\{t\geq0; X_{t+ g}=0\}$ and for all stopping time $T<\infty$,
$$f(A_{T+ g}-A_{ g})X_{T+ g}=\E[f(A_{\infty})X_{\infty}1_{\{L\leq T\}}|\mathcal{F}_{T+g}].$$
\end{enumerate}
In both cases, we obtain:
$$\rho(f(A_{\cdot+ g}-A_{ g})X_{\cdot+ g})_{T}=\rho(\E[f(A_{\infty})X_{\infty}1_{\{L\leq \cdot\}}|\mathcal{F}_{\cdot+g}])_{T}.$$
Then, we proceed in the same way as in Theorem \ref{t1} to conclude the proof.

\end{proof}

\subsection{Representation formulas of processes of class \texorpdfstring{$\sum_{s}(H)$}{sigma(s)(H)}}

We shall begin this subsection by giving the following lemma which gives a representation formula of a uniformly integrable $\Qv-$ martingale null on $H$.
\begin{lem}\label{l0}
Let $M$ be a uniformly integrable $(\Qv,(\mathcal{F}_{t}))-$ martingale which vanishes on $H$. Then, on $\{ g<\infty\}$ and for every stopping time $T$, the following holds:
\begin{equation}
	M_{T}=\E[M_{\infty}1_{\{ g\leq T\}}|\mathcal{F}_{T}].
\end{equation}
\end{lem}
\begin{proof}
By definition, $DM$ is a uniformly integrable $\Pv-$ martingale with respect to $(\mathcal{F}_{t})_{t\geq0}$ and vanishes on $H$. Hence, proceeding in the same way as in the proof of Theorem \ref{qt}, we get that $M_{\cdot+ g}$ is a uniformly integrable  martingale with respect to $(\mathcal{F}_{t+g})_{t\geq0}$. Then, for $M_{\infty}=\lim_{t\to\infty}{M_{t+ g}}$, we have:
$$M_{T+ g}=\E[M_{\infty}|\mathcal{F}_{T+g}],$$
for every stopping time $T$. Now, remark that the process defined by \\$Y_{t}=\E[M_{\infty}1_{\{ g\leq t\}}|\mathcal{F}_{t}]$ vanishes on $H$. Furthermore, for every $t\geq0$,
$$Y_{t+ g}=\E[M_{\infty}|\mathcal{F}_{t+g}].$$
Then, $\forall$ $t\geq0$,
$$Y_{t}=\rho(\E[M_{\infty}|\mathcal{F}_{\cdot+g}])_{t}.$$
Consequently, for every stopping time $T$,
$$M_{T}=\E[M_{\infty}1_{\{ g\leq T\}}|\mathcal{F}_{T}].$$
This completes the proof.

\end{proof}

Next Corollary is a direct application of Lemma \ref{l0} and Theorem 3.3 of \cite{f}.
\begin{coro}
Let $X$ be a process of $\sum_{s}(H)$ class and $f:\R_{+}\longrightarrow\R_{+}$ be a locally bounded Borel function. Then $M^{f}_{t}=f(A_{t})X_{t}-_{\Qv}\int_{0}^{t}{f(A_{s})dA_{s}}$ is a uniformly integrable $\Qv-$ martingale. Furthermore, on $\{g<\infty\}$, we have:
\begin{equation}
f(A_{T})X_{T}-_{\Qv}\int_{0}^{T}{f(A_{s})dA_{s}}=\E[M_{\infty}^{f}1_{\{ g\leq T\}}|\mathcal{F}_{T}],
\end{equation}
for every stopping time T.
\end{coro}

Now, we shall establish representation formulas of processes of $\sum_{s}(H)$ class. We get these results in two distinct ways. In the first one, we exploit the relationship between $\sum_{s}(H)$ class and the class $\mathcal{R_{+}}$ of \cite{1}. In the second one, we shall exploit the link with class $\sum$. Let us recall the definition of the class $\mathcal{R_{+}}$.
\begin{defn}
Let $Y$ be a non-negative process. We will say that $Y$ is a process of class $\mathcal{R_{+}}$ if:
\begin{itemize}
	\item the random set $\{t\geq0; Y_{t}=0\}$ is closed;
	\item $Y$ admits the decomposition of the form $Y=N+A$, where $N$ is a right continuous martingale which is uniformly integrable and $A$ is a continuous nondecreasing process which is adapted and integrable such that $dA_{t}$ is carried by $\{t\geq0; Y_{t}=0\}$;
	\item $\Pv[Y_{\infty}=0]=0$.
\end{itemize}
\end{defn}

Now, we give the relationship between $\sum_{s}(H)$ and $\mathcal{R_{+}}$.
\begin{lem}\label{l2}
Let $X$ be a stochastic process of $\sum_{s}(H)$ class such that $\Pv[X_{\infty}=0]=0$. Then, the stochastic process $X_{\cdot+ g}$ is an element of class $\mathcal{R}_{+}$.
\end{lem}
\begin{proof}
Let $X=N+A$ be a stochastic process of $\sum_{s}(H)$ class. Since $N$ is a right continuous, uniformly integrable $(\Qv, (\mathcal{F}_{t})_{t\geq0})-$ martingale. It follows by Quotient Theorem that $N_{\cdot+ g}$ is a right continuous, uniformly integrable $\Pv-$ martingale with respect to the filtration $(\mathcal{F}_{t+g})_{t\geq0}$. Furthermore, by definition, $A_{\cdot+ g}$ is a continuous nondecreasing process such that $dA_{t+ g}$ is carried by $\{t\geq0; X_{t+ g}=0\}$. This achieves the proof.

\end{proof}

The following theorem gives one of the main results of this subsection.
\begin{theorem}
Let $X$ be an integrable stochastic process of $\sum_{s}(H)$ class such that $\Pv[X_{\infty}=0]=0$. Let $\Gamma=\sup\{t\geq g; X_{t}=0\}$. Then, for every stopping time $T$,
\begin{equation}
	X_{T}=\E[X_{\infty}1_{\{\Gamma\leq T\}}|\mathcal{F}_{T}].
\end{equation}
\end{theorem}
\begin{proof}
Lemma \ref{l2} allows to claim that $X_{\cdot+ g}$ is a stochastic process of class $\mathcal{R}_{+}$ with respect to the filtration $(\mathcal{F}_{t+g})_{t\geq0}$. Hence, for every stopping time $T$, one has:
$$X_{T+ g}=\E[X_{\infty}1_{\{L\leq T\}}|\mathcal{F}_{T+g}].$$
This implies that
$$\rho(X_{\cdot+ g})_{T}=\rho(\E[X_{\infty}1_{\{L\leq \cdot\}}|\mathcal{F}_{\cdot+g}])_{T}.$$
Since $X$ vanishes on $H$, it follows that,
$$X_{T}=\rho(\E[X_{\infty}1_{\{L\leq \cdot\}}|\mathcal{F}_{\cdot+g}])_{T}.$$
Now, let $Y_{t}=\E[X_{\infty}1_{\{ g+L\leq t\}}|\mathcal{F}_{t}]$. Remark that $Y$ vanishes on $H$. Furthermore, one has
$$Y_{t+ g}=\E[X_{\infty}1_{\{L+ g\leq t+ g\}}|\mathcal{F}_{t+g}]$$
$$\hspace{0.25cm}=\E[X_{\infty}1_{\{L\leq t\}}|\mathcal{F}_{t+g}].$$
Then, 
$$Y_{t}=\rho(\E[X_{\infty}1_{\{L\leq \cdot\}}|\mathcal{F}_{\cdot+g}])_{t}.$$
Consequently, $X\equiv Y$ since those are càdlàg processes. Therefore, for every stopping time $T$,
$$X_{T}=\E[X_{\infty}1_{\{ g+L\leq T\}}|\mathcal{F}_{T}].$$

\end{proof}

Now, we shall exploit a relationship between  the $\sum_{s}(H)$ class and the class $\sum$ to prove the last main results of this section.  
\begin{theorem}\label{com}
Let $X$ be a process of $\sum_{s}(H)$ class and $f:\R\longrightarrow\R$ be a Borel function. Let \\$\Gamma=\sup{\{t\geq0; X_{t}=0\}}$. Then, the following assertions hold:
\begin{enumerate}
	\item If $\widetilde{X}$ is of class $\Dv$, then there exist integrable random variables $X_{\infty}$, $N_{\infty}$, $A_{\infty}$ such that $\widetilde{X}_{t}\longrightarrow X_{\infty}$, $\widetilde{N}_{t}\longrightarrow N_{\infty}$, $\widetilde{A}_{t}\longrightarrow A_{\infty}$ as $t\rightarrow\infty$ almost surely as well as in $L^{1}$ and for all stopping time $T$
	\begin{equation}
	f(A_{T})X_{T}=\E[f(A_{\infty})X_{\infty}1_{\{\Gamma\leq T\}}|\mathcal{F}_{T}].
	\end{equation}
	\item If $q:\R\longrightarrow\R-\{0\}$ is a Borel function such that $q(\widetilde{A})\widetilde{X}$ is of class $\Dv$, then there exist random variables $X_{\infty}$, $N_{\infty}$, $A_{\infty}$ such that $\widetilde{X}_{t}\longrightarrow X_{\infty}$, $\widetilde{N}_{t}\longrightarrow N_{\infty}$, $\widetilde{A}_{t}\longrightarrow A_{\infty}$ as $t\rightarrow\infty$ almost everywhere on $\{L< \infty\}$ and
	\begin{equation}
	f(A_{T})X_{T}=\E[f(A_{\infty})X_{\infty}1_{\{\Gamma\leq T\}}|\mathcal{F}_{T}];
	\end{equation}
	for all stopping time $T<\infty$.
\end{enumerate}
In particular, in both cases one has
\begin{equation}
X_{T}=\E[X_{\infty}1_{\{\Gamma\leq T\}}|\mathcal{F}_{T}].
\end{equation}
\end{theorem}
\begin{proof}
From Proposition 3.2 of \cite{f}, $\widetilde{X}$ is of class $\big(\sum\big)$. Then, according to Theorem 3.1 of \cite{pat}, the following hold:
\begin{enumerate}
	\item If $\widetilde{X}$ is of class $\Dv$, then there exist integrable random variables $X_{\infty}$, $N_{\infty}$, $A_{\infty}$ such that $\widetilde{X}_{t}\longrightarrow X_{\infty}$, $\widetilde{N}_{t}\longrightarrow N_{\infty}$, $\widetilde{A}_{t}\longrightarrow A_{\infty}$  as $t\rightarrow\infty$ almost surely, as well as in $L^{1}$, and
	$$f(\widetilde{A}_{T})\widetilde{X}_{T}=\E[f(A_{\infty})X_{\infty}1_{\{L\leq T\}}|\mathcal{F}_{g+T}].$$
	\item If $q:\R\longrightarrow\R-\{0\}$ is a Borel function such that $q(\widetilde{A})\widetilde{X}$ is of class $\Dv$, then there exist random variables $X_{\infty}$, $N_{\infty}$, $A_{\infty}$ such that $\widetilde{X}_{t}\longrightarrow X_{\infty}$, $\widetilde{N}_{t}\longrightarrow N_{\infty}$, $\widetilde{A}_{t}\longrightarrow A_{\infty}$  as $t\rightarrow\infty$ almost everywhere on $\{L< \infty\}$, and
	$$f(\widetilde{A}_{T})\widetilde{X}_{T}=\E[f(A_{\infty})X_{\infty}1_{\{L\leq T\}}|\mathcal{F}_{g+T}].$$
\end{enumerate}
In particular, in both cases, one has
$$\widetilde{X}_{T}=\E[X_{\infty}1_{\{L\leq T\}}|\mathcal{F}_{g+T}].$$
Now, define $Y_{t}:=\E[f(A_{\infty})X_{\infty}1_{\{L+ g\leq t\}}|\mathcal{F}_{t}]$ and $Z_{t}=\E[X_{\infty}1_{\{L+ g\leq T\}}|\mathcal{F}_{T}]$. We have for every $t\geq0$,
$$Y_{t+ g}=\E[f(A_{\infty})X_{\infty}1_{\{L\leq t\}}|\mathcal{F}_{t+g}]=f(A_{t+ g})X_{t+ g};$$
and
$$Z_{t+ g}=\E[X_{\infty}1_{\{L\leq t\}}|\mathcal{F}_{t+g}]=X_{t+ g}.$$
Furthermore, the processes $X$, $Y$ and $Z$ vanish on $H$. Then, in both cases, it follows from Proposition 3.1 of \cite{1} that:
$$f(A_{T})X_{T}=\E[f(A_{\infty})X_{\infty}1_{\{L+ g\leq T\}}|\mathcal{F}_{T}]$$
and
$$X_{T}=\E[X_{\infty}1_{\{L+ g\leq T\}}|\mathcal{F}_{T}].$$
This achieves the proof.

\end{proof}

\begin{theorem}\label{t02}
Let $X$ be a process of $\sum_{s}(H)$ class and $f:\R\longrightarrow\R$ be a Borel function. There exist random variables $X_{\infty}$, $N_{\infty}$, $A_{\infty}$ such that $\widetilde{X}_{t}\longrightarrow X_{\infty}$, $\widetilde{N}_{t}\longrightarrow N_{\infty}$, $\widetilde{A}_{t}\longrightarrow A_{\infty}$ almost everywhere on $\{L<\infty\}$ and
\begin{equation}
f(A_{T})X_{T}=\E[f(A_{\infty})X_{\infty}1_{\{\Gamma\leq T\}}|\mathcal{F}_{T}],
\end{equation}
for all stopping time $T<\infty.$ In particular,
\begin{equation}
X_{T}=\E[X_{\infty}1_{\{\Gamma\leq T\}}|\mathcal{F}_{T}],
\end{equation}
for all stopping time $T<\infty$.
\end{theorem}
\begin{proof}
From Proposition 3.2 of \cite{f}, $\widetilde{X}$ is of class $\big(\sum\big)$. Since $N$ is a uniformly integrable $\Qv-$ martingale, it follows from Quotient Theorem of \cite{1} that $\widetilde{N}$ is a uniformly integrable martingale. Then according to Corollary 3.2 of \cite{pat}, there exist random variables $X_{\infty}$, $N_{\infty}$, $A_{\infty}$ such that $\widetilde{X}_{t}\longrightarrow X_{\infty}$, $\widetilde{N}_{t}\longrightarrow N_{\infty}$, $\widetilde{A}_{t}\longrightarrow A_{\infty}$ almost everywhere on $\{L<\infty\}$ and
$$f(\widetilde{A}_{T})\widetilde{X}_{T}=\E[f(A_{\infty})X_{\infty}1_{\{L\leq T\}}|\mathcal{F}_{g+T}],$$ 
for all stopping time $T<\infty.$ In particular,
$$\widetilde{X}_{T}=\E[X_{\infty}1_{\{L\leq T\}}|\mathcal{F}_{g+T}]$$
for all stopping time $T<\infty$. Then we proceed in the same way as in Theorem \ref{com} to conclude.

\end{proof}
\begin{rem}
If $X$ satisfies the assumptions of Theorem \ref{t02}, then there exists a random variable $X_{\infty}$ such that $X_{t}\longrightarrow X_{\infty}$ almost everywhere on $\{L<\infty\}$, and one has
\begin{equation}
X_{t}=\E[X_{\infty}1_{\{\Gamma\leq t\}}|\mathcal{F}_{t}],
\end{equation}
for all $t\geq0$.
\end{rem}
\begin{theorem}\label{m}
Let $X$ be a process of $\sum_{s}(H)$ class and $f:\R_{+}\longrightarrow\R_{+}$ be a locally bounded Borel function such that the process $f(\widetilde{A})\widetilde{X}$ is of class $\Dv$ and $f(\widetilde{A}_{t})\widetilde{X}_{t}\longrightarrow1$ almost surely. Let $F(x)=\int_{0}^{x}{f(y)dy}$ and $\Gamma=\sup{\{t\geq g; X_{t}=0\}}$.
\begin{enumerate}
	\item If $F(\infty)<\infty$, then, $\forall t\geq0$, $A_{t}=0$ and $f(0)X_{t}=1_{H^{c}}(t)$.
	\item If $F(\infty)=\infty$, then $L<\infty$, $\widetilde{A}_{L}=A_{\infty}<\infty$ and $X_{t}\longrightarrow X_{\infty}$ almost surely for some random variable $X_{\infty}>0$. Moreover, for every stopping time $T$, one has
\end{enumerate}
\begin{equation}\label{36}
f(A_{T})X_{T}=\Pv[ \Gamma\leq T|\mathcal{F}_{T}].
\end{equation}
\end{theorem}
\begin{proof}
From Proposition 3.2 of \cite{f}, it follows that $\widetilde{X}$ is a process of class $\sum$. Then thanks to Theorem 3.10 of \cite{pat} ( or Theorem 3.8 of \cite{pat1}), we have
\begin{itemize}
	\item If $F(\infty)<\infty$, then $\widetilde{A}_{t}=0$ and $f(0)\widetilde{X}_{t}=1$ for all $t\geq0$.
	\item If $F(\infty)=\infty$, then $L<\infty$, $\widetilde{A}_{L}=A_{\infty}<\infty$ and $\widetilde{X}_{t}\longrightarrow X_{\infty}$ almost surely for a random variable $X_{\infty}>0$. Moreover, for every stopping time $T$ one has
	$$f(\widetilde{A}_{T})\widetilde{X}_{T}=\Pv[L\leq T|\mathcal{F}_{g+T}].$$
\end{itemize}
Since $A$ and $X$ vanish on $H$, it follows that:
\begin{itemize}
	\item if $F(\infty)<\infty$, then $A_{t}=\rho(\widetilde{A}_{\cdot})_{t}=0$ and $f(0)X_{t}=1_{H^{c}}(t)$ for all $t\geq0$;
	\item if  $F(\infty)=\infty$, then $L<\infty$, $\widetilde{A}_{L}=A_{\infty}<\infty$ and $\widetilde{X}_{t}\longrightarrow X_{\infty}$ almost surely for some random variable $X_{\infty}>0$. Moreover, for every stopping time $T$ one has
	$$f(A_{T})X_{T}=\rho(\Pv[L\leq \cdot|\mathcal{F}_{g+\cdot}])_{T}.$$ 
\end{itemize}
If we take $Y_{t}=\Pv[L+ g\leq t|\mathcal{F}_{t}]$, $Y$ will vanish on $H$, and for all $t\geq0$, $Y_{t+ g}=\Pv[L\leq t|\mathcal{F}_{g+t}]$. Then, by uniqueness, 
$$f(A_{T})X_{T}=\Pv[ g+L\leq T|\mathcal{F}_{T}].$$
This completes the proof.

\end{proof}
\begin{theorem}
Let $X$ be a process of $\sum_{s}(H)$ class and $f:\R_{+}\longrightarrow\R_{+}$ be a  Borel function for which there exists an increasing sequence $(a_{n})_{n\in \mathbb{N}}$ in $(0,\infty)$ such that $f1_{[0,a_{n}]}$ is bounded for all $n\in\mathbb{N}$, and $f(x)=0$ for all $x\geq a:=\lim_{n\rightarrow\infty}{a_{n}}$. Let $F(x)=\int_{0}^{x}{f(y)dy}$ and assume that the process $f(\widetilde{A})\widetilde{X}$ is of class $\Dv$ and $f(\widetilde{A}_{t})\widetilde{X}_{t}\longrightarrow1$ almost surely. Let $\Gamma=\sup{\{t\geq g; X_{t}=0\}}$. Then,
\begin{enumerate}
	\item If $F(a)<\infty$, then $A_{t}=0$ and $f(0)X_{t}=1_{H^{c}}(t)$, $\forall t\geq0$.
	\item If $F(a)=\infty$, then $L<\infty$, $\widetilde{A}_{L}=A_{\infty}<a$ and $X_{t}\longrightarrow X_{\infty}$ almost surely for some random variable $X_{\infty}>0$. Moreover, for every stopping time $T$ one has 
\begin{equation}\label{99}
f(A_{T})X_{T}=\Pv[ \Gamma\leq T|\mathcal{F}_{T}].
\end{equation}	
\end{enumerate}
\end{theorem}
\begin{proof}
From Proposition 3.2 of \cite{f}, it follows that $\widetilde{X}$ is a process of class $\sum$. Then thanks to Theorem 3.9 of \cite{pat1}, we have
\begin{itemize}
	\item If $F(a)<\infty$, then $\widetilde{A}_{t}=0$ and $f(0)\widetilde{X}_{t}=1$ for all $t\geq0$.
	\item If $F(a)=\infty$, then $L<\infty$, $\widetilde{A}_{L}=A_{\infty}<a$ and $\widetilde{X}_{t}\longrightarrow X_{\infty}$ almost surely for some random variable $X_{\infty}>0$. Moreover, for every stopping time $T$, one has
	$$f(\widetilde{A}_{T})\widetilde{X}_{T}=\Pv[L\leq T|\mathcal{F}_{g+T}].$$
\end{itemize}
Then, we proceed in the same way as in Theorem \ref{m} to conclude the proof.

\end{proof}

\begin{acknowledgements}
We thank the referee for the careful reading of the paper and for his valuable remarks. This work is supported by
\begin{itemize}
	\item Hassan II Academy of Sciences and Technology, project 'Mathematics and applications'.
	\item Agence Nationale des Bourses du Gabon (ANBG).
\end{itemize} 
\end{acknowledgements}

{\color{myaqua}

}}

\end{document}